\documentclass[11pt,a4paper]{amsart}
\usepackage{graphicx,multirow,array,amsmath,amssymb}

\newtheorem{theorem}{Theorem}

\newtheorem{proposition}[theorem]{Proposition}

\begin{document}

\title{Minimum lattice length and ropelength of knots}
\author[K. Hong]{Kyungpyo Hong}
\address{Department of Mathematics, Korea University, Anam-dong, Sungbuk-ku, Seoul 136-701, Korea}
\email{cguyhbjm@korea.ac.kr}
\author[H. Kim]{Hyoungjun Kim}
\address{Department of Mathematics, Korea University, Anam-dong, Sungbuk-ku, Seoul 136-701, Korea}
\email{kimhjun@korea.ac.kr}
\author[S. No]{Sungjong No}
\address{Department of Statistics, Ewha Womans University, 52, Ewhayeodae-gil, Seodaemun-gu, Seoul 120-750, Korea}
\email{sungjongno84@gmail.com}
\author[S. Oh]{Seungsang Oh}
\address{Department of Mathematics, Korea University, Anam-dong, Sungbuk-ku, Seoul 136-701, Korea}
\email{seungsang@korea.ac.kr}

\thanks{This research was supported by Basic Science Research Program through
the National Research Foundation of Korea(NRF) funded by the Ministry of Science,
ICT \& Future Planning(MSIP) (No.~2011-0021795).}
\thanks{This work was supported by the BK21 Plus Project through the National Research Foundation 
of Korea (NRF) funded by the Korean Ministry of Education (22A20130011003).}

\begin{abstract}
Let $\mbox{Len}(K)$ be the minimum length of a knot on the cubic lattice 
(namely the minimum length necessary to construct the knot in the cubic lattice).
This paper provides upper bounds for $\mbox{Len}(K)$ of a nontrivial knot $K$
in terms of its crossing number $c(K)$ as follows:

\vspace{2mm}
\hspace{5mm} $\mbox{Len}(K) \leq \min \left\{ \frac{3}{4}c(K)^2 + 5c(K) + \frac{17}{4}, \,
\frac{5}{8}c(K)^2 + \frac{15}{2}c(K) + \frac{71}{8} \right\}.$
\vspace{2mm}

The ropelength of a knot is the quotient of its length by its thickness,
the radius of the largest embedded normal tube around the knot.
We also provide upper bounds for the minimum ropelength $\mbox{Rop}(K)$
which is close to twice $\mbox{Len}(K)$:

\vspace{2mm}
\hspace{5mm} $\mbox{Rop}(K) \leq \min \left\{ \hspace{-1mm}
\begin{tabular}{l}  
$1.5 c(K)^2  + 9.15 c(K) + 6.79$, \\  
$1.25 c(K)^2  + 14.58 c(K) + 16.90$ 
\end{tabular}
\hspace{-1mm} \right\}.$

\end{abstract}

\maketitle

\section{Introduction} \label{sec:intro}

A knot can be embedded in many different ways in $3$-space, smooth or piecewise linear.
Polygonal knots are those which consist of finite line segments, called {\em sticks\/},
attached end-to-end.
A {\em lattice knot} is a polygonal knot in the cubic lattice
$\mathbb{Z}^3=(\mathbb{R} \times \mathbb{Z} \times
\mathbb{Z}) \cup (\mathbb{Z} \times \mathbb{R} \times \mathbb{Z})
\cup (\mathbb{Z} \times \mathbb{Z} \times \mathbb{R})$.
For further studies on lattice knots the readers are referred to
\cite{ACCJSZ, D1, HNO1, HO1, HO2}.

A quantity that we may naturally be interested on lattice knots is
the minimum length necessary to realize a knot as a lattice knot.
An {\em edge\/} is a line segment of unit length joining two nearby lattice points in $\mathbb{Z}^3$.
Obviously a stick with length $n$ consists of $n$ edges.
The minimum number of edges necessary to realize a knot $K$ as a lattice knot
is called the {\em minimum lattice length\/}, denoted by $\mbox{Len}(K)$.

Diao \cite{D1} introduced this term (he used ``minimal edge number" instead),
and proved that the minimum lattice length of the trefoil knot $3_1$ is 24
and all the other nontrivial knots need more than 24 edges.
Ishihara and Shimokawa (see \cite{SIADSV}) proved that the minimum lattice length of $4_1$ and $5_1$
are $30$ and $34$, respectively.
Lattice knot presentations with minimum lattice length of the knots $3_1$, $4_1$,
and $5_1$ are depicted in Figure \ref{fig1}.
Also many numerical estimations of the minimum lattice length for various knots
are addressed in \cite{HNRAV, JaP, SIADSV}.

\begin{figure}[h]
\includegraphics[scale=1]{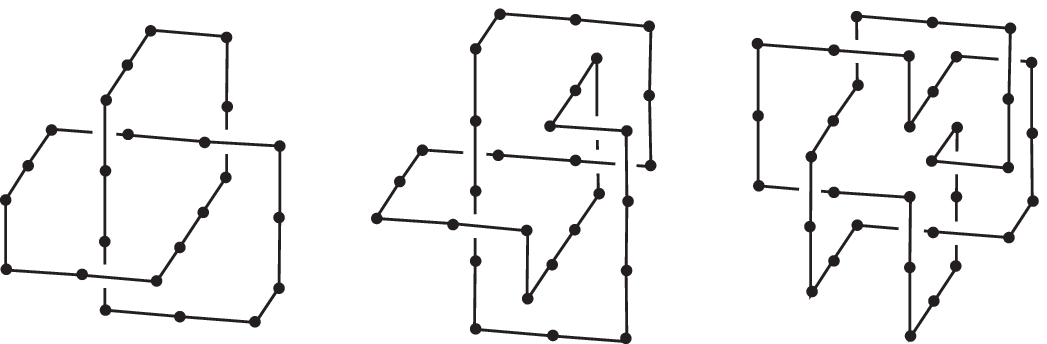}
\caption{Examples of $3_1$, $4_1$ and $5_1$ with minimum lattice length}
\label{fig1}
\end{figure}

In the paper we are interested in finding upper bounds for the minimum lattice length of
a given nontrivial knot $K$ in terms of the minimum crossing number $c(K)$.
Diao {\em et al\/} \cite{DEY} established an O$(c(K)^{\frac{3}{2}})$ upper bound
for the minimum lattice length as follows:

\vspace{2mm}
\hspace{5mm} $\mbox{Len}(K) \leq 136 c(K)^{\frac{3}{2}} + 84 c(K) + 22 c(K)^{\frac{1}{2}} + 11.$
\vspace{2mm}

Recently the authors \cite{HNO2} obtained another inequality for nontrivial knots
except the trefoil knot $3_1$:

\vspace{2mm}
\hspace{5mm} $\mbox{Len}(K) \leq \frac{3}{2}c(K)^2 + 2c(K) + \frac{1}{2}.$
\vspace{2mm}

In this paper we improve this upper bound.

\begin{theorem} \label{thm:Len}
Let $K$ be any nontrivial knot.

\vspace{2mm}
\hspace{5mm} ${\rm Len}(K) \leq \min \left\{ \frac{3}{4}c(K)^2 + 5c(K) + \frac{17}{4}, \,
\frac{5}{8}c(K)^2 + \frac{15}{2}c(K) + \frac{71}{8} \right\}.$
\vspace{2mm}

Moreover if $K$ is a non-alternating prime knot,

\vspace{2mm}
\hspace{5mm} ${\rm Len}(K) \leq \min \left\{ \frac{3}{4}c(K)^2  + 2c(K) - \frac{11}{4}, \,
\frac{5}{8}c(K)^2  + 5c(K) - \frac{29}{8} \right\}.$
\end{theorem}

An essential question in physical knot theory concerns the minimum
amount of rope (of unit thickness) needed to tie a given knot.
We measure the {\em ropelength\/} of a knot $K$ as the quotient of its length by its thickness
where thickness is the radius of the largest embedded normal tube around the knot.
We define the minimum ropelength $\mbox{Rop}(K)$ to be the minimum ropelength of all embeddings of $K$.
Many results about finding lower bounds for the ropelength of $K$ as a function of
crossing number can be found in \cite{BS, CKS, D2, DE}.
In this paper we are interested in a converse problem:
finding upper bounds for the minimum ropelength of $K$.
Indeed, the minimum ropelength of a knot is less than or equal to twice the minimum lattice length of the knot.
In the same paper, Diao {\em et al\/} also established an O$(c(K)^{\frac{3}{2}})$ upper bound
for the minimum ropelength which is twice the upper bound for the minimum lattice length they found:

\vspace{2mm}
\hspace{5mm} $\mbox{Rop}(K) \leq 272 c(K)^{\frac{3}{2}} + 168 c(K) + 44 c(K)^{\frac{1}{2}} + 22.$
\vspace{2mm}

Also Cantarella {\em et al\/} \cite{CFM} obtained another inequality:

\vspace{2mm}
\hspace{5mm} $\mbox{Rop}(K) \leq 1.64 c(K)^2 + 7.69 c(K) + 6.74.$
\vspace{2mm}

We obtained an improved inequality from Theorem \ref{thm:Len}.

\begin{theorem} \label{thm:Rop}
Let $K$ be any nontrivial knot.

\vspace{2mm}
\hspace{5mm} ${\rm Rop}(K) \leq \min \left\{  \hspace{-1mm}
\begin{tabular}{l}  
$\frac{3}{2}c(K)^2  + (\pi + 6) c(K) + 2 \pi + \frac{1}{2}$, \\  
$\frac{5}{4}c(K)^2  + (\frac{\pi}{2} + 13) c(K) + \pi + \frac{55}{4}$ 
\end{tabular}
\hspace{-1mm} \right\}.$
\vspace{2mm}

Moreover if $K$ is a non-alternating prime knot,

\vspace{2mm}
\hspace{5mm} ${\rm Rop}(K) \leq \min \left\{ \hspace{-1mm}
\begin{tabular}{l}  
$\frac{3}{2}c(K)^2  + \pi c(K) - \frac{11}{2}$, \\  
$\frac{5}{4}c(K)^2  + (\frac{\pi}{2} + 8) c(K) - \frac{29}{4}$ 
\end{tabular}
\hspace{-1mm} \right\}.$
\end{theorem}

Note that the first inequality in Theorem \ref{thm:Rop} guarantees the following inequality
which slightly improves the Cantarella's bound:

\vspace{2mm}
\hspace{5mm} $\mbox{Rop}(K) \leq \min \left\{ \hspace{-1mm}
\begin{tabular}{l}  
$1.5 c(K)^2  + 9.15 c(K) + 6.79$, \\  
$1.25 c(K)^2  + 14.58 c(K) + 16.90$ 
\end{tabular}
\hspace{-1mm} \right\}.$

\section{Grid diagrams and arc index} \label{sec:grid}

A {\em grid diagram\/} of a knot $K$ is a knot diagram of vertical strands and
the same number of horizontal strands with the properties that
at every crossing the vertical strand crosses over the horizontal strand and
no two horizontal strands are co-linear and no two vertical strands are co-linear.
It is known that every knot admits a grid diagram \cite{C}.
The minimum number of vertical strands in all grid diagrams of $K$
is called the {\em grid index\/} of K, denoted by $g(K)$.

An {\em arc presentation\/} of a knot $K$ is an embedding of $K$
in finitely many pages of an open-book decomposition
so that each of the pages meets $K$ in a single simple arc.
Note that this open-book decomposition has open half-planes as pages
and the standard $z$-axis as the binding axis.
And the {\em arc index\/} $\alpha(K)$ is the minimum number of pages among all
possible arc presentations of $K$.

Since a grid diagrams is a way of depicting an arc presentation \cite{C},
the arc index equals the grid index, i.e. $\alpha(K)=g(K)$.
Figure \ref{fig2} presents an arc presentation, a grid diagram, and how they are related
in an example of the trefoil knot.

\begin{figure} [h]
\includegraphics[scale=0.8]{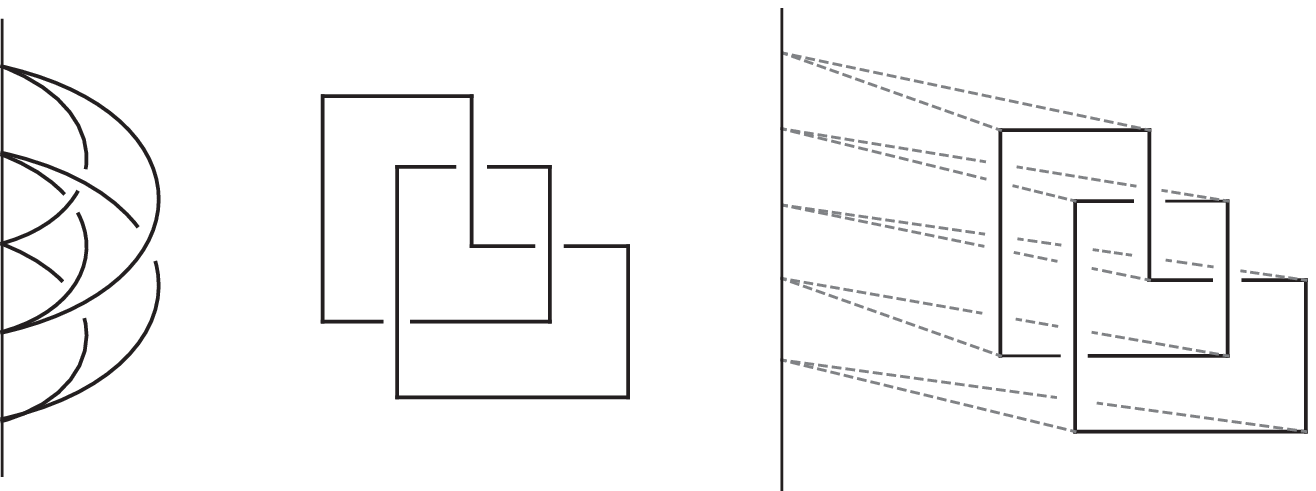}
\caption{An arc presentation and a grid diagram}
\label{fig2}
\end{figure}

In this paper we first find upper bounds for the minimum lattice length and
the minimum ropelength of knots in terms of the grid index.
The following proposition has a key role to convert these bounds to
upper bounds in terms of the crossing number.

\begin{proposition} \label{prop:grid}
Let $K$ be any nontrivial knot. Then $g(K) \leq c(K)+2$.
Moreover if $K$ is a non-alternating prime knot, then $g(K) \leq c(K)$.
\end{proposition}

\begin{proof}
Bae and Park \cite{BP} established an upper bound on arc index in terms of the crossing number.
In their paper, Corollary $4$ and Theorem $9$ provide that $\alpha(K) \leq c(K)+2$,
and moreover $\alpha(K) \leq c(K)+1$ if $K$ is a non-alternating prime knot.
Later Jin and Park \cite{JiP} improved the second part of Bae and Park's theorem.
They proved $a(K) \leq c(K)$ for a non-alternating prime knot $K$.
The equality $\alpha(K)=g(K)$ guarantees the proposition.
\end{proof}

A stick in $\mathbb{Z}^3$ parallel to the $x$-axis is called an {\em $x$-stick\/} and
an edge parallel to the $x$-axis is called an {\em $x$-edge\/}.
The plane with the equation $x=i$ for some integer $i$ is called an {\em $x$-level $i$}.
The related notations concerning the $y$ and $z$-coordinates
will be defined in the same manner as the $x$-coordinate.
Note that each $y$-stick or $z$-stick lies on some $x$-level.

\section{Upper bounds for the minimum lattice length} \label{sec:Len}

In this section we prove Theorem \ref{thm:Len}
by actually constructing a lattice knot and counting the number of its edges.
Let $K$ be a nontrivial knot with the grid index $g$.
The proof follows four steps:
settling a grid diagram of $K$ into $\mathbb{Z}^3$,
folding it twice horizontally and vertically to reduce the number of edges,
and apply Proposition \ref{prop:grid}.

\vspace{3mm}
\noindent {\bf Step 1.}
{\em Natural settlement of a grid diagram of $K$ into $\mathbb{Z}^3$.\/}
\vspace{2mm}

We begin by considering a grid diagram of $K$
which consists of $g$ horizontal strands and $g$ vertical strands.
We realize this grid diagram into the cubic lattice $\mathbb{Z}^3$ as follows.
We regard these $g$ horizontal strands as $x$-sticks lying on $z$-level 1
and $y$-levels $1, 2, \cdots , g$ in the same order as they appeared in the original grid diagram.
Similarly we regard the other $g$ vertical strands as $y$-sticks lying on $z$-level 2
and $x$-levels $1, 2, \cdots , g$.
Now connect each pair of an $x$-stick and a $y$-stick, which came from adjacent strands
in the grid diagram, by a $z$-edge whose boundary lies on $z$-levels 1 and 2.
We finally get a lattice presentation of $K$ as shown in the leftmost of Figure \ref{fig3}.

Now we count the number of edges.
By the definition of the grid diagram,
there are exactly two $x$-edges each between $x$-levels 1 and 2, and between $x$-levels $g-1$ and $g$, 
four $x$-edges each between $x$-levels 2 and 3, and between $x$-levels $g-2$ and $g-1$, and so on.
Every time we pass over each $x$-level, we add at most two $x$-edges
until we reach the middle level
because the number of $x$-edges increases by two usually, but sometimes is unchanged or decreases by two.
Thus the maximum number of $x$-edges is
$2 \sum^{\frac{g-1}{2}}_{n=1} 2 n = \frac{g^2-1}{2}$ if $g$ is odd, or
$2 \sum^{\frac{g}{2}-1}_{n=1} 2 n + 2 (\frac{g}{2})= \frac{g^2}{2}$ if $g$ is even.
We count the number of $y$-edges in the same manner.
Obviously the number of $z$-edges is equal to $2g$.
Therefore the maximum number of edges of this lattice presentation is
$\frac{g^2-1}{2} + \frac{g^2-1}{2} + 2g = g^2  + 2g - 1$ if $g$ is odd, or
$\frac{g^2}{2} + \frac{g^2}{2} + 2g = g^2  + 2g$ if $g$ is even.
So we have the following inequality:

\vspace{2mm}
\hspace{5mm} $\mbox{Len}(K) \leq g^2 + 2g.$

\vspace{3mm}
\noindent {\bf Step 2.}
{\em Horizontal folding to reduce one fourth of the square growth.\/}
\vspace{2mm}

We would like to fold this lattice presentation structure along the line $z=1$, $x=\frac{g+1}{2}$ if $g$ is odd, 
and $z=1$, $x=\frac{g}{2} +1$ if $g$ is even  as shown in Figure \ref{fig3}.
During this folding argument, any stick does not pass another but
only some $x$-sticks can be overlapped.
To get a proper lattice presentation, remove such all overlapped $x$-edges
which are expressed by dotted line segments in the last figure.
Note that almost half of $y$-sticks which are on the right side in the leftmost figure moved into $z$-level 0
and all $x$-sticks stayed on $z$-level 1.
This movement does not change the knot type of $K$.

\begin{figure} [h]
\includegraphics[scale=1]{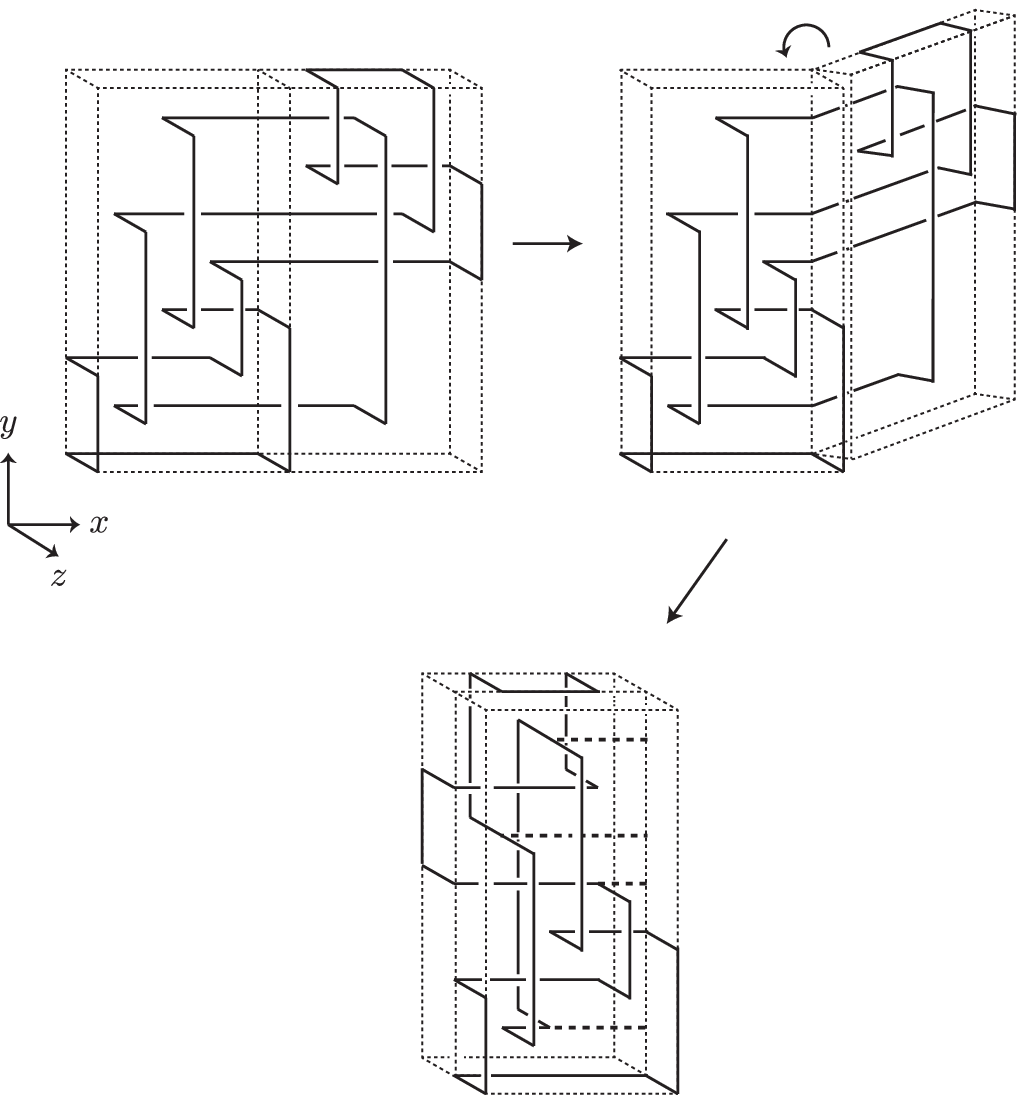}
\caption{Horizontal folding}
\label{fig3}
\end{figure}

Again we count the number of edges.
Note that the numbers of $y$-edges and $z$-edges are unchanged.
If $g$ is odd, then each $x$-level $1, \cdots , \frac{g-1}{2}$ contains exactly two $y$-sticks
and $x$-level $\frac{g+1}{2}$ contains one $y$-stick.
This means that there are four $x$-edges between $x$-levels 1 and 2,
at most eight $x$-edges between $x$-levels 2 and 3, and so on.
Similarly there are two $x$-edges between $x$-levels $\frac{g-1}{2}$ and $\frac{g+1}{2}$,
at most six $x$-edges between $x$-levels $\frac{g-3}{2}$ and $\frac{g-1}{2}$, and so on.
Keep counting until we reach the middle level.
Then the maximum number of $x$-edges is
$\sum^{\frac{g-1}{2}}_{n=1} 2 n = \frac{g^2-1}{4}$.
Moreover we can reduce two more $z$-edges in $x$-level $\frac{g+1}{2}$.
If $g$ is even, then each $x$-level $2, \cdots , \frac{g}{2}-1$ contains two $y$-sticks
and each $x$-level 1 and $\frac{g}{2}+1$ contains one $y$-stick.
Thus there are two $x$-edges each between $x$-levels 1 and 2,
and between $x$-levels $\frac{g}{2}$ and $\frac{g}{2}+1$,
at most six $x$-edges each between $x$-levels 2 and 3,
and between $x$-levels $\frac{g}{2}-1$ and $\frac{g}{2}$, and so on.
Then the maximum number of $x$-edges is
$2 \sum^{\frac{g}{4}}_{n=1} (4n-2) = \frac{g^2}{4}$ if $\frac{g}{2}+1$ is odd, or
$2 \sum^{\frac{g}{4}-\frac{1}{2}}_{n=1} (4n-2) + 4(\frac{g}{4} + \frac{1}{2}) -2 = \frac{g^2}{4} +1$
if $\frac{g}{2}+1$ is even.
Moreover we can reduce two $z$-edges each in $x$-levels 1 and $\frac{g}{2}+1$.

Therefore the maximum number of edges of this new lattice presentation is
$\frac{g^2-1}{4} + \frac{g^2-1}{2} + (2g-2) = \frac{3}{4}g^2  + 2g - \frac{11}{4}$ if $g$ is odd,
$\frac{g^2}{4} + \frac{g^2}{2} + (2g-4) = \frac{3}{4}g^2  + 2g - 4$
if $g=4k$ for some positive integer $k$, or
$(\frac{g^2}{4}+1) + \frac{g^2}{2} + (2g-4) = \frac{3}{4}g^2  + 2g - 3$
if $g=4k+2$.
So we have the following inequality:

\vspace{2mm}
\hspace{5mm} $\mbox{Len}(K) \leq \frac{3}{4}g^2  + 2g - \frac{11}{4}.$

\vspace{3mm}
\noindent {\bf Step 3.}
{\em Vertical folding to reduce one eighth of the square growth.\/}
\vspace{2mm}

We fold again the new lattice presentation structure along the line $z=2$, $y=\frac{g+1}{2}$ if $g$ is odd, 
and $z=2$, $y=\frac{g}{2} +1$ if $g$ is even as shown in Figure \ref{fig4}.
When we fold, all $y$-sticks on $z$-level 2 stayed on the same $z$-level,
upper half of $x$-sticks on $z$-level 1 moved into $z$-level 3 and
the rest half stayed on $z$-level 1,
and upper half of $y$-sticks on $z$-level 0 moved into $z$-level 4 and
the rest half stayed on $z$-level 0.
During this folding argument, a stick does not pass another,
only some of $y$-sticks on $z$-level 2 are overlapped,
and some of $y$-sticks on $z$-level 0 are broken into two pieces.
To get a proper lattice presentation, remove such all overlapped $y$-edges,
and connect each pair of the broken $y$-sticks by two $y$-edges and four $z$-edges 
as shown in the last figure in Figure \ref{fig4}.
Still this movement does not change the knot type of $K$.

\begin{figure} [h]
\includegraphics[scale=1]{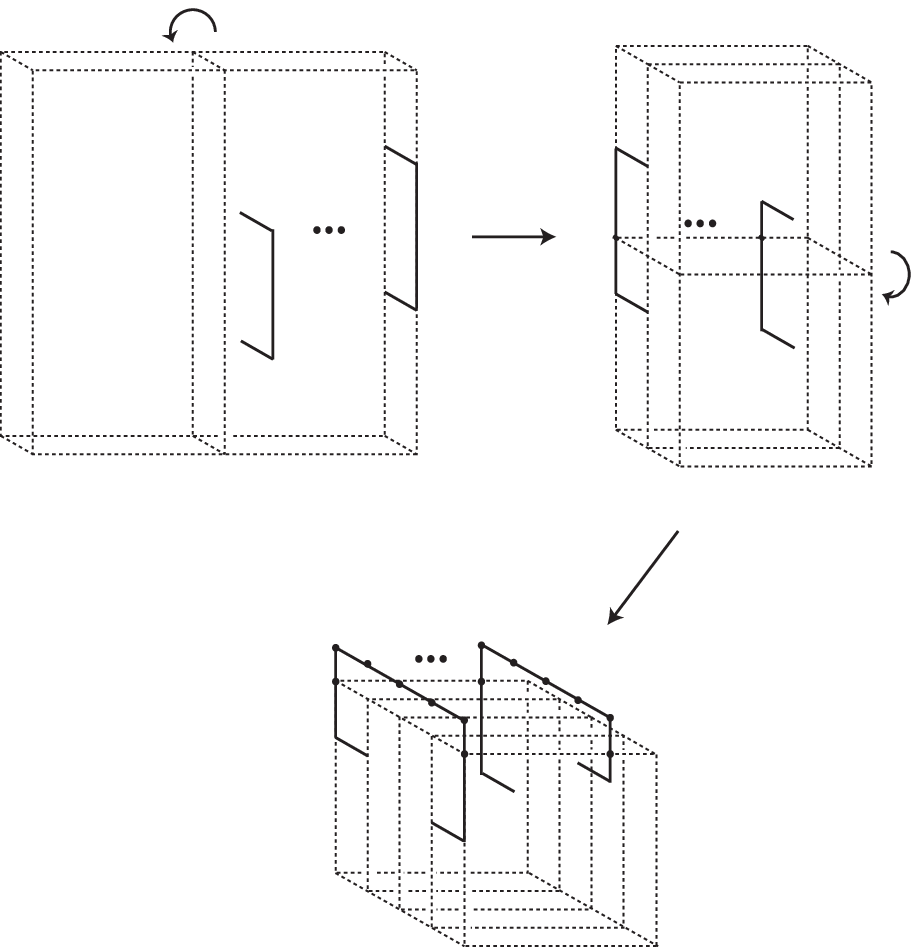}
\caption{Vertical folding}
\label{fig4}
\end{figure}

In this step, counting the number of edges is more complicated.
First consider the $y$-sticks on $z$-level 2 which are at least half among $g$ $y$-sticks.
Indeed, the exact number of such $y$-sticks is $\frac{g+1}{2}$ for odd $g$ case,
and $\frac{g}{2}+1$ for even $g$ case.
Folding these $y$-sticks is exactly in the same way as folding $x$-sticks in Step 2.
When we fold $x$-sticks in Step 2,
the maximum number of $x$-edges decreases by almost half as follows:
$\frac{g^2-1}{2}$ replaced by $\frac{g^2-1}{4}$ for odd $g$,
$\frac{g^2}{2}$ replaced by $\frac{g^2}{4}$ for $g=4k$, and
$\frac{g^2}{2}$ replaced by $\frac{g^2}{4}+1$ for $g=4k+2$.
So we can apply this calculation to the above half class of $y$-sticks on $z$-level 2.
More precisely the decrease by almost half is guaranteed 
when we fold all $y$-sticks as proceeded in Step 2.
But we can surely apply this argument to either 
the set of almost half $y$-sticks on the left side
or the set of almost half $y$-sticks on the right side in the first figure in Figure \ref{fig3}.  
We may assume that the above half class of $y$-sticks on $z$-level 2 fits this condition.

Note that for the other half class of $y$-sticks on $z$-levels 0 and 4,
we add two more $y$-edges and four more $z$-edges to connect each pair after broken.
Then the maximum number of $y$-edges is
$\frac{1}{2} (\frac{g^2-1}{4}) + \frac{1}{2} (\frac{g^2-1}{2}) + 2 (\frac{g-1}{2})
= \frac{3}{8}g^2 + g - \frac{11}{8}$ for odd $g$ case,
$\frac{1}{2} (\frac{g^2}{4}) + \frac{1}{2} (\frac{g^2}{2}) + 2 (\frac{g}{2}-1)
= \frac{3}{8}g^2 + g - 2$ for $g=4k$ case, and
$\frac{1}{2} (\frac{g^2}{4}+1) + \frac{1}{2} (\frac{g^2}{2}) + 2 (\frac{g}{2}-1)
= \frac{3}{8}g^2 + g - \frac{3}{2}$ for $g=4k+2$ case.

After counting the $z$-sticks used to connect each pair of the broken $y$-sticks,
the maximum number of $z$-edges in total is
$2g + 4 (\frac{g-1}{2}) = 4g - 2$ for odd $g$, and
$2g + 4 (\frac{g}{2}-1) = 4g - 4$ for even $g$,
while the number of $x$-edges are unchanged from Step 2.
Therefore the maximum number of edges of the final lattice presentation is
$\frac{g^2-1}{4} + (\frac{3}{8}g^2 + g - \frac{11}{8}) + (4g-2)
= \frac{5}{8}g^2  + 5g - \frac{29}{8}$ if $g$ is odd,
$\frac{g^2}{4} + (\frac{3}{8}g^2 + g - 2) + (4g-4)
= \frac{5}{8}g^2  + 5g - 6$
if $g=4k$, or
$(\frac{g^2}{4}+1) + (\frac{3}{8}g^2 + g - \frac{3}{2}) + (4g-4)
= \frac{5}{8}g^2  + 5g - \frac{9}{2}$
if $g=4k+2$.
We have the following inequality:

\vspace{2mm}
\hspace{5mm} $\mbox{Len}(K) \leq \frac{5}{8}g^2  + 5g - \frac{29}{8}.$

\vspace{3mm}
\noindent {\bf Step 4.}
{\em Apply Proposition \ref{prop:grid} to complete the proof of Theorem \ref{thm:Len}.\/}
\vspace{2mm}

Since we can take any upper bounds of Step 1, 2 and 3, we finally get:

\vspace{2mm}
\hspace{5mm} $\mbox{Len}(K) \leq \min \left\{\frac{3}{4}g^2  + 2g - \frac{11}{4}, \,
\frac{5}{8}g^2  + 5g - \frac{29}{8} \right\}.$
\vspace{2mm}

Now apply Proposition \ref{prop:grid} to complete the proof of Theorem \ref{thm:Len}.

\section{Upper bounds for the minimum ropelength}

In this section we prove Theorem \ref{thm:Rop}.
Since a lattice knot can be changed to a smooth knot
by taking a twice enlargement and replacing every corner with a suitable quartile circle
as illustrated in Figure \ref{fig5}.
Apply this argument to each steps in Section \ref{sec:Len}.
This means that we take twice the length and subtract $2 - \frac{\pi}{2}$ the number of corners times.

\begin{figure} [h]
\includegraphics[scale=1]{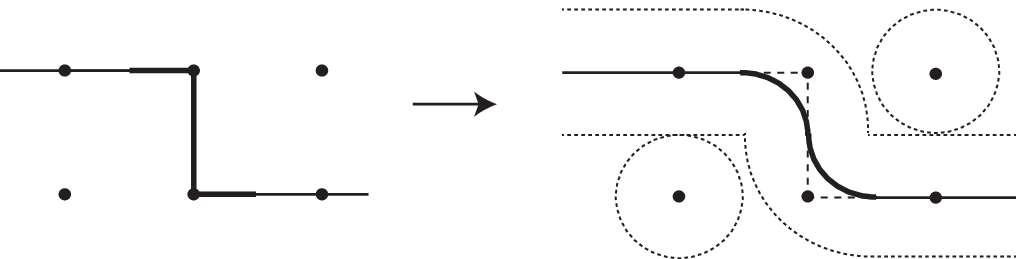}
\caption{Replacing every corner with a quartile circle}
\label{fig5}
\end{figure}

In Step 1, the number of corners is $4g$.
So we have the inequality
$\mbox{Rop}(K) \leq 2(g^2 + 2g) + 4g(\frac{\pi}{2} - 2) = 2g^2 + (2 \pi - 4)g.$

In Step 2, at least $2g$ corners most of which are placed on $z$-levels 0 and 2 are guaranteed.
Thus $\mbox{Rop}(K) \leq 2(\frac{3}{4}g^2  + 2g - \frac{11}{4}) + 2g(\frac{\pi}{2} - 2)
= \frac{3}{2}g^2  + \pi g - \frac{11}{2}.$

In Step 3, at least $g$ corners which are placed on $z$-levels 0 and 4 are guaranteed.
Thus $\mbox{Rop}(K) \leq 2(\frac{5}{8}g^2  + 5g - \frac{29}{8}) + g(\frac{\pi}{2} - 2)
= \frac{5}{4}g^2  + (\frac{\pi}{2} + 8) g - \frac{29}{4}.$

Therefore we finally get the following inequality:

\vspace{2mm}
\hspace{5mm} $\mbox{Rop}(K) \leq \min \left\{ \frac{3}{2}g^2  + \pi g - \frac{11}{2}, \,
\frac{5}{4}g^2  + (\frac{\pi}{2} + 8) g - \frac{29}{4} \right\}.$
\vspace{2mm}

Again apply Proposition \ref{prop:grid} to complete the proof of Theorem \ref{thm:Rop}.

\end{document}